\documentclass[a4paper,12pt]{article}
\usepackage[T1]{fontenc}
\usepackage[latin1]{inputenc}
\usepackage[english]{babel}
\usepackage{amsthm}
\usepackage{amsmath}
\usepackage{amssymb}
\usepackage{amsfonts}
\usepackage{enumerate}

\usepackage{array,lmodern,graphicx}

\def\Re{\mathfrak{Re}}
\def\Im{\mathfrak{Im}}
\def\R{\mathbb{R}}

\def\a{\mathfrak{a}}

\newtheorem{prop}{Proposition}[section]

\newtheorem{teo}{Theorem}[section]
\newtheorem{df}{Definition}[section]

\def\N{\Arrowvert}
\def\n{\arrowvert}

\def\e{\varepsilon}

\def\i{\infty}

\def\aa{\alpha}

\def\na{\nabla}
\newcommand{\D}{\Delta}

\title{Riesz transforms associated to Schr\"{o}dinger operators with negative potentials}
\author{Joyce ASSAAD\footnote{Universit\'e Bordeaux 1, Institut de Math\'ematiques (IMB). CNRS UMR 5251.Equipe d'Analyse et G\'eom\'etrie. 351 Cours de la Liberation 33405 Talence, France.
 Tel: (33) 05 40 00 21 71, Fax: (33) 05 40 00 26 26. e-mail adress: joyce.assaad@math.u-bordeaux1.fr}}
\date{}

\begin{document}
\maketitle

\begin{abstract}
The goal of this paper is to study the Riesz transforms $\na A^{-1/2}$ where $A$ is the Schr\"{o}dinger operator $-\D-V,\ \ V\ge 0$, under different conditions on the potential $V$. We prove that if $V$ is strongly subcritical, $\na A^{-1/2}$ is bounded on $L^p(\R^N)$ , $N\ge3$, for all $p\in(p_0';2]$ where $p_0'$ is the dual exponent of  $p_0$ where $2<\frac{2N}{N-2}<p_0<\i$; and we give a counterexample to the boundedness on $L^p(\R^N)$ for $p\in(1;p'_0)\cup(p_{0*};\i)$ where $p_{0*}:=\frac{p_0N}{N+p_0}$ is the reverse Sobolev exponent of $p_0$. 
If the potential is strongly subcritical in the Kato subclass $K_N^{\i}$, then $\na A^{-1/2}$ is bounded on $L^p(\R^N)$ for all $p\in(1;2]$, moreover if it is in $L^{N/2}(\R^N)$ then $\na A^{-1/2}$ is bounded on $L^p(\R^N)$ for all $p\in(1;N)$. We prove also boundedness of $V^{1/2}A^{-1/2}$ with the same conditions on the same spaces. Finally we study these  operators on manifolds. We prove that our results hold on a class of Riemannian manifolds.\\
\textbf{keywords:}Riesz transforms, Schr\"odinger operators,  off-diagonal estimates,  singular operators, Riemannian manifolds.
\end{abstract}
\section{Introduction and definitions}
Let $A$ be a Schr\"{o}dinger operator $-\D+V$ where $-\D$ is the nonnegative Laplace operator and the potential $V:\R^N\rightarrow \R$ such that $V=V^{+}-V^{-}$ (where $V^{+}$ and $V^{-}$ are the positive and negative parts of $V$, respectively). The operator is defined via the sesquilinear form method. We define
\[\a(u,v)=\int_{\R^N}\na u(x)\na v(x)dx+\int_{\R^N} V^{+}(x)u(x)v(x)dx-\int_{\R^N} V^{-}(x)u(x)v(x)dx\]
\[D(\a)=\left\{u\in W^{1,2}(\R^N), \int_{\R^N}V^{+}(x)u^2(x)dx<\i\right\}.\]
Here we assume $V^{+}\in L^1_{loc}(\R^N)$ and $V^-$ satisfies (for all $u\in D(\a)$):
\begin{eqnarray}\label{klmn}
\int_{\R^N}&&V^-(x)u^2(x)dx \le\nonumber\\
&&\aa\left[\int_{\R^N}|\na u|^2(x)dx+\int_{\R^N}V^{+}(x)u^2(x)dx\right]+\beta\int_{\R^N}u^2(x)dx
\end{eqnarray}
where $\aa \in(0,1)$ and $\beta \in \R$. By the well-known KLMN theorem (see for example \cite{k} Chapter VI), the form $\a$ is closed (and bounded from below). Its associated operator is $A$. If in addition $\beta\le 0$, then $A$ is nonnegative.

 We can define the Riesz transforms associated to $A$ by
\[\na A^{-1/2}:=\frac{1}{\Gamma(\frac{1}{2})}\int_0^{\i}\sqrt{t}\na e^{-tA}\frac{dt}{t}.\]

 The boundedness of Riesz transforms on $L^p(\R^N)$ implies that the domain of $A^{1/2}$ is included in the Sobolev space $W^{1,p}(\R^N)$. Thus the solution of the corresponding evolution equation will be in the Sobolev space $W^{1,p}(\R^N)$ for initial data in $L^p(\R^N)$.\\

It is our aim to study the boundedness on $L^p(\R^N)$ of the Riesz transforms $\na A^{-1/2}$.
We are also interested in the boundedness of the operator $V^{1/2}A^{-1/2}$. If $\na A^{-1/2}$ and  $V^{1/2}A^{-1/2}$ are bounded on $L^p(\R^N)$, we obtain for some positive constant $C$
\[\N \na u\N_p+\N V^{1/2}u\N_p\le C \N (-\D+V)^{1/2}u\N_p.\]
By a duality argument, we obtain
\[\N (-\D+V)^{1/2}u\N_{p'}\le C(\N \na u\N_{p'}+\N V^{1/2}u\N_{p'})\]
where $p'$ is the dual exponent of $p$.\\
Riesz transforms associated to Schr\"{o}dinger operators with nonnegative potentials were studied by   Ouhabaz \cite{o}, Shen \cite{sh}, and  Auscher and Ben Ali \cite{aba}. Ouhabaz proved that Riesz transforms are bounded on $L^p(\R^N)$ for all $p\in (1;2]$, for all potential $V$ locally integrable. Shen and Auscher and Ben Ali proved that if the potential $V$ is in the reverse H\"{o}lder class $B_q$, then the Riesz transforms are bounded on $L^p(\R^N)$ for all $p\in (1,p_1)$ where $2<p_1\le \i$ depends on $q$. The result of Auscher and Ben Ali generalize that of Shen because Shen has restrictions on the dimension $N$ and on the class $B_q$. Recently, Badr and Ben Ali \cite{bba} extend the result of Auscher and Ben Ali \cite{aba} to Riemannian manifolds of homogeneous type with polynomial volume growth where Poincar\'e inequalities hold and Riesz transforms associated to the Laplace-Beltrami operator are bounded. They also prove that a smaller range is possible if the volume
  growth is not polynomial.\\

With negative potentials new difficulties appear. If we take $V\in L^{\i}(\R^N)$, and  apply the method in \cite{o} to the operator $A+\N V\N_{\i}$, we obtain boundedness of $\na (A+\N V\N_{\i})^{-1/2}$ on $L^p(\R^N)$ for all $p\in(1;2]$. This is weaker than the boundedness of $\na A^{-1/2 }$ on the same spaces. Guillarmou and Hassell \cite{gh1} studied Riesz transforms $\na (A\circ P_+)^{-1/2 }$ where $A$ is the Schr\"{o}dinger operator with negative potential and $P_+$ is the spectral projection on the positive spectrum.
They prove that, on asymptotically conic manifolds $M$ of dimension $N\ge3$, if $V$ is smooth and satisfies decay conditions, and  the Schr\"{o}dinger operator has no zero-modes nor zero-resonances, then Riesz transforms $\na (A\circ P_+)^{-1/2}$
are bounded on $L^p(M)$ for all $p\in(1, N)$. They also prove (see \cite{gh2}) that when zero-modes are present, Riesz transforms $\na (A\circ P_+)^{-1/2 }$  are bounded on $L^p(M)$ for all $p\in(\frac{N}{N-2}, \frac{N}{3})$, with bigger range possible if the zero modes have extra decay at infinity.\\

In this paper we consider only negative potentials. From now on, we denote by $A$ the Schr\"{o}dinger operator with negative potential,
 \[A:=-\D-V,\ \  V\ge 0.\]
 Our purpose is, first, to find optimal conditions on $V$ allowing the boundedness of Riesz transforms $\na A^{-1/2}$ and that of $V^{1/2}A^{-1/2}$ on $L^p(\R^N)$ second, to find  the best possible range of $p$'s. \\
 Let us take the following definition
\begin{df}
We say that the potential $V$ is strongly subcritical if for some $\e >0$, $A\ge \e V$. This means that for all $u\in W^{1,2}(\R^N)$
\[\int_{\R^N} Vu^2\le \frac{1}{1+\e}\int_{\R^N}|\na u|^2.\]
\end{df}
For more information on strongly subcritical potentials see \cite{ds} and \cite{z}.\\
With this condition, $V$ satisfies assumption (\ref{klmn}) where $\beta=0$ and $\aa=\frac{1}{1+\e}$.  Thus $A$ is well defined, nonnegative and $-A$ generates an analytic contraction semigroup $(e^{-tA})_{t\ge0}$ on $L^2(\R^N)$.\\

Since $-\D-V\ge \e V$ we have $(1+\e)(-\D-V)\ge \e(-\D)$. Therefore
\begin{eqnarray}\label{l2}
 ||\na u||_2^2\le (1+\frac{1}{\e})||A^{1/2}u||_2^2.
\end{eqnarray}
Thus, $\na A^{-1/2}$ is bounded on $L^2(\R^N)$. Conversely, it is clear that if $\na A^{-1/2}$ is bounded on $L^2(\R^N)$ then $V$ is strongly subcritical.\\

We observe also that  $-\D-V\ge \e V$ is equivalent to
\begin{eqnarray}\label{lv2}
 ||V^{1/2} u||_2^2\le \frac{1}{\e}||A^{1/2}u||_2^2.
\end{eqnarray}
 Thus, $V^{1/2} A^{-1/2}$ is bounded on $L^2(\R^N)$ if and only if $V$ is strongly subcritical.\\

So we can conclude that
\[\N \na u\N_2+\N V^{1/2}u\N_2\le C\N (-\D-V)^{1/2}u\N_2\]
if and only if $V$ is strongly subcritical.
Then by duality argument we have
\[\N \na u\N_2+\N V^{1/2}u\N_2\approx \N (-\D-V)^{1/2}u\N_2\]
if and only if $V$ is strongly subcritical.\\

To study Riesz transforms on $L^p(\R^N)$ for $1\le p\le \i$ with $p\neq 2$ we use the results on the uniform boundedness of the semigroup on $L^p(\R^N)$. Taking central potentials which are equivalent to $c/|x|^2$ as $|x|$ tends to infinity where $0<c<(\frac{N-2}{2})^2, N\ge3$, Davies and Simon  \cite{ds} proved that for all $t>0$ and all $p\in (p'_0;p_0)$,
\[\N e^{-tA}\N_{p-p}\le C   \]
 where $2<\frac{2N}{N-2}<p_0
<\i$ and $p_0'$ its dual exponent. Next Liskevich, Sobol, and Vogt \cite{LSV} proved the uniform boundedness on $L^p(\R^N)$ for all $p\in (p_0';p_0)$ where $2<\frac{2N}{N-2}<p_0=\frac{2N}{(N-2)\big(1-\sqrt{1-\frac{1}{1+\e}}\big)}$, for general strongly subcritical potentials . They also proved that the range $(p_0',p_0)$ is optimal and the semigroup  does not even act on $L^p(\R^N)$ for $p\notin(p_0',p_0)$. Under additional condition on $V$, Takeda \cite{t} used stochastic methods to prove a Gaussian estimate of the associated heat kernel. Thus the semigroup acts boundedly on $L^p(\R^N)$ for all $p\in[1,\i]$. \\

In this paper we prove that when $V$ is strongly subcritical and $N\ge3$, Riesz transforms are bounded on $L^p(\R^N)$ for all $p\in (p_0';2]$.
 We also give a counterexample to the boundedness of Riesz transforms on $L^p(\R^N)$  when $p\in(1;p'_0)\cup(p_{0*};\i)$ where $2<p_{0*}:=\frac{p_0N}{N+p_0}<p_0<\i$.
  If $V$ is strongly subcritical in the Kato subclass $K_N^\i, N\ge3$ (see Section \ref{tk}), then  $\na A^{-1/2}$ is bounded on $L^p(\R^N)$ for all $p\in(1,2]$.
  If, in addition, $V\in L^{N/2}(\R^N)$ then it is bounded on $L^p(\R^N)$ for all $p\in(1,N).$ With the same conditions, we prove  similar results for the operator $V^{1/2}A^{-1/2}$. Hence if $V$ is strongly subcritical and $V\in K_N^\i\cap L^{N/2}(\R^N), N\ge 3$, then
\begin{eqnarray}
\N \na u\N_p+\N V^{1/2}u\N_p\approx \N (-\D-V)^{1/2}u\N_p
\end{eqnarray}
for all 
$p\in (N';N)$.\\
In the last section, we extend our results to Riemannian manifolds. We denote by $-\D$ the Laplace-Beltrami operator on a complete non-compact Riemannian manifold $M$ of dimension $N\ge 3$. We prove that when $V$ is strongly subcritical on $M$, $\na (-\D-V)^{-1/2}$  and $V^{1/2}(-\D-V)^{-1/2}$ are bounded on $L^p(M)$ for all $p\in (p'_0;2]$ if $M$ is of homogeneous type and the Sobolev inequality holds on $M$.
 If in addidtion  Poincar\'e inequalities hold on $M$ and $V$ belongs to the Kato class $K_\i$ then $\na (-\D-V)^{-1/2}$ and $V^{1/2}(-\D-V)^{-1/2}$ are bounded on $L^p(M)$ for all $p\in (1;2]$. When  
    $V$ is in addition in $L^{N/2}(M)$ and the Riesz transforms associated to the Laplace-Beltrami operator are bounded on $L^r(M)$ for some $r\in (2;N]$, then $\na (-\D-V)^{-
 1/2}$ and $V^{1/2}(-\D-V)^{-1/2}$ are bounded on $L^p(M)$ for all $p\in (1;r)$.\\

For the proof of the boundedness of Riesz transforms we use  off-diagonal estimates (for properties and more details see \cite{am}). These estimates are a generalization of the Gaussian estimates used by Coulhon and Duong in \cite{cduong} to study the Riesz transforms associated to the Laplace-Beltrami operator on Riemannian manifolds, and by Duong, Ouhabaz and Yan in \cite{doy} to study the magnetic Schr\"{o}dinger operator on $\R^N$.  We also use the approach of Blunck and Kunstmann in \cite{BK} and \cite{bk3} to weak type $(p,p)$-estimates. In \cite{a}, Auscher used these tools to divergence-form operators with complex coefficients.
For $p\in (2;N)$ we use a complex interpolation method (following an idea in Auscher and Ben Ali \cite{aba}).\\
In contrast to \cite{gh1} and \cite{gh2}, we do not assume decay nor smoothness conditions on $V$.\\

In the following sections, we denote by $L^p$ the Lebesgue space $L^p(\R^N)$ with the Lebesgue measure $dx$, $|| . ||_p$ its usual norm, $(.,.)$ the inner product of $L^2$,  $|| . ||_{p-q}$ the norm of operators acting from $L^p$ to $L^q$. We denote by $p'$ the dual exponent to $p$, $p':=\frac{p}{p-1}$. We denote by $C, c$ the positive constants even if their values change at each occurrence. Through this paper, $\na A^{-1/2}$ denotes one of  the partial derivative $\frac{\partial}{\partial x_k}A^{-1/2}$ for any fixed $k\in \{1,...,N\}$.

 \section{Off-diagonal estimates}

In this section, we show that $(e^{-tA})_{t>0}, (\sqrt{t} \na e^{-tA})_{t>0}$ and $(\sqrt{t} V^{1/2} e^{-tA})_{t>0}$ satisfy  $L^p-L^2$ off-diagonal estimates provided that $V$ is strongly subcritical.\\

 \begin{df}\label{def}
Let $(T_t)_{t>0}$ be a family of uniformly bounded operators on $L^2$. We say that $(T_t)_{t>0}$ satisfies $L^p-L^q$ off-diagonal estimates for  $p,q \in [1;\infty]$ with $p\le q$ if there exist positive constants $C$ and $c$ such that for all closed sets $E$ and $F$ of $\R^N$ and all $h\in L^p(\R^N)\cap L^2(\R^N)$with support in $E$, we have for all $t>0$:
\[ \N T_th\N_{L^q(F)}\le Ct^{-\gamma_{pq}}e^{-\frac{cd(E,F)^2}{t}}\N h\N_p,\]
where $d$ is the Euclidean distance and $\gamma_{pq}:=\frac{N}{2}\big(\frac{1}{p}-\frac{1}{q}\big)$.
 \end{df}

\begin{prop}\label{22}
Let $A=-\Delta-V$ where $V\ge0$ and $V$ is strongly subcritical. Then $(e^{-tA})_{t>0}$,  $(\sqrt{t}\na e^{-tA})_{t>0}$, and $(\sqrt{t} V^{1/2} e^{-tA})_{t>0}$ satisfy  $L^2-L^2$ off-diagonal estimates, and we have for all $t >0$ and all $f \in L^2$ supported in $E$:
\begin{enumerate}
\item[(i)] $|| e^{-tA}f ||_{L^2(F)}\le  e^{-d^2(E,F)/4t} || f||_2,$ 
\item[(ii)] $|| \sqrt{t}\na e^{-tA}f ||_{L^2(F)}\le C e^{-d^2(E,F)/16t} || f||_2, $
\item[(iii)] $|| \sqrt{t}V^{1/2} e^{-tA}f ||_{L^2(F)}\le C e^{-d^2(E,F)/8t} || f||_2. $
\end{enumerate}
\end{prop}
\begin{proof}[Proof:]
The ideas are classical and rely on the well known  Davies perturbation technique. Let $A_\rho:= e^{\rho \phi}Ae^{-\rho \phi}$ where $\rho >0$ and $\phi$ is a Lipschitz function with $|\nabla \phi |\le 1$ a.e.. Here $A_{\rho}$ is the associated operator to the sesquilinear form $\a_{\rho}$ defined by
 \[\a_{\rho}(u,v):=\a(e^{-\rho \phi}u, e^{\rho \phi}v)\]
for all $u,v \in D(\a)$.

By the strong subcriticality property of $V$ we have for all $u\in W^{1,2}$
\begin{eqnarray}
 ((A_\rho+\rho^2)u,u)&=& -\int \rho^2|\na \phi|^2u^2+\int |\na u|^2-\int Vu^2+\rho^2||u||^2_2\nonumber\\
&\ge& \e \N V^{1/2}u\N^2_2\label{rop}.
 \end{eqnarray}
Using (\ref{l2}), we obtain
\begin{eqnarray}
 ((A_\rho+\rho^2)u,u)&=& -\int \rho^2|\na \phi|^2u^2+\int |\na u|^2-\int Vu^2+\rho^2||u||^2_2\nonumber\\
&\ge& \frac{\e}{\e+1} ||\nabla u||_2^2\label{ro}.
 \end{eqnarray}
  In particular $(A_\rho+\rho^2)$ is a maximal accretive operator on $L^2$, and this implies
\begin{eqnarray}\label{ro1}
 ||e^{-tA_{\rho}}u||_2\le e^{t\rho^2}||u||_2.
\end{eqnarray}
Now we want to estimate
\[\N (A_\rho+2\rho^2)e^{-t(A_\rho+2\rho^2)}\N_{2-2}.\]
First, let us prove that $A_\rho +2\rho^2$ is a sectorial operator.

For $u$ complex-valued,
\[\a_{\rho}(u,u):=\a(u,u)+ \rho \int u\na \phi \overline{\na u}-\rho \int\overline{u} \na \phi \na u- \rho^2 \int \n \na \phi \n^2 \n u\n^2.\]
Then
\begin{eqnarray*}
 \a_{\rho}(u,u)+2\rho^2\N u\N^2_2&\ge&\a(u,u)+ \rho \int u\na \phi \overline{\na u}-\rho \int\overline{u} \na \phi \na u+ \rho^2 \N u\N_2^2\\
&=& \a(u,u) +2i\rho \Im\int u \na \phi \overline{\na u}+ \rho^2 \N u\N_2^2.
\end{eqnarray*}
This implies that
\begin{eqnarray}\label{re}
 \Re(\a_{\rho}(u,u)+2\rho^2\N u\N^2_2)\ge \a(u,u),
\end{eqnarray}
and
\begin{eqnarray}\label{reel}
 \Re(\a_{\rho}(u,u)+2\rho^2\N u\N^2_2)\ge \rho^2\N u\N^2_2.
\end{eqnarray}

On the other hand,
\begin{eqnarray*}
\a_{\rho}(u,u)&=&\a(u,u)+ \rho \int u\na \phi \overline{\na u}-\rho \int\overline{u} \na \phi \na u- \rho^2 \int \n \na \phi \n^2 \n u\n^2\\
&=& \a(u,u)+2i \rho \Im \int u\na \phi \overline{\na u}- \rho^2 \int \n \na \phi \n^2 \n u\n^2.
\end{eqnarray*}
So
\begin{eqnarray*}
|\Im(\a_{\rho}(u,u)+2\rho^2\N u\N^2_2)|&=& 2|\rho| \int |u| |\na \phi| |\overline{\na u}|\\
&\le&2|\rho| \N u\N_2 \N \na u\N_2.
\end{eqnarray*}
Using (\ref{l2}) we obtain that
\begin{eqnarray*}
|\Im(\a_{\rho}(u,u)+2\rho^2\N u\N^2_2)|&\le& 2|\rho| \N u\N_2 c_\e\a^{\frac{1}{2}}(u,u)\\
&\le&c_\e^2\a(u,u)+\rho^2\N u\N_2^2,
\end{eqnarray*}
where $c_\e=(1+\frac{1}{\e})^{\frac{1}{2}}$.
Now using estimates (\ref{re}) and (\ref{reel}), we deduce that there exists a constant $C>0$ depending only on $\e$ such that
\[|\Im(\a_{\rho}(u,u)+2\rho^2\N u\N^2_2)|\le C \Re(\a_{\rho}(u,u)+2\rho^2\N u\N^2_2).\]
We conclude that (see \cite{k} or \cite{o})
$$\N e^{-z(A_\rho +2\rho^2)}\N_{2-2}\le 1$$for all $z$ in the open sector of angle $\arctan(1/C)$. Hence by the Cauchy formula
\begin{eqnarray}\label{gr}
\N (A_\rho+2\rho^2)e^{-t(A_\rho+2\rho^2)}\N_{2-2}\le \frac{C}{t}.
\end{eqnarray}
The constant $C$ is independent of $\rho$.\\
By estimate (\ref{rop}) and (\ref{ro}) we have
\[((A_\rho+2\rho^2)u,u)\ge ((A_\rho+\rho^2)u,u)\ge \e\N V^{1/2} u\N_2^2,\]
and
\[((A_\rho+2\rho^2)u,u)\ge ((A_\rho+\rho^2)u,u)\ge \frac{\e}{\e+1}\N \na u\N_2^2.\]
 Setting $u=e^{-t(A_\rho+2\rho^2)}f$ and using  (\ref{gr}) and (\ref{ro1})

we obtain

  \begin{eqnarray}\label{na}
  ||\sqrt{t}\nabla e^{-tA_{\rho}}f||_2\le Ce^{2t\rho^2}||f||_2.
  \end{eqnarray} and

\begin{eqnarray}\label{vro}
||\sqrt{t}V^{1/2} e^{-tA_{\rho}}f||_2\le Ce^{2t\rho^2}||f||_2.
\end{eqnarray}
 Let E and F be two closed subsets of $\R^N$, $f\in L^2(\R^N)$ supported in E, and let $\phi(x):=d(x,E)$ where $d$ is the Euclidean distance. Since $e^{\rho\phi}f=f$,  we have the following relation
 \[e^{-tA}f=e^{-\rho \phi}e^{-tA_{\rho}}f.\]
 Then \[ \na e^{-tA}f=-\rho \na \phi e^{-\rho \phi}e^{-tA_{\rho}}f+e^{-\rho \phi}\na e^{-tA_{\rho}}f,\]
and
\[V^{1/2}e^{-tA}f=e^{-\rho \phi}V^{1/2}e^{-tA_{\rho}}f.\]
 Now taking the norm on $L^2(F)$, we obtain from (\ref{ro1}), (\ref{na}) and (\ref{vro})
 \begin{eqnarray}\label{ediag}
 || e^{-tA} f||_{L^2(F)}\le  e^{-\rho d(E,F)} e^{\rho^2t}|| f||_2,
 \end{eqnarray}
 \begin{eqnarray}\label{nadiag}
 ||\na e^{-tA}f ||_{L^2(F)}\le \rho  e^{- \rho d(E,F)} e^{\rho^2t}|| f||_2+ \frac{C}{\sqrt{t}} e^{- \rho d(E,F)}e^{2t\rho^2}||f||_2,
 \end{eqnarray}
 and
\begin{eqnarray}\label{vdiag}
 ||V^{1/2} e^{-tA} f||_{L^2(F)}\le \frac{C}{\sqrt{t}} e^{-\rho d(E,F)} e^{2\rho^2t}|| f||_2.
 \end{eqnarray}
We set $\rho=d(E,F)/2t$ in (\ref{ediag}) and $\rho=d(E,F)/4t$ in (\ref{vdiag}), then we get the $L^2-L^2$ off-diagonal estimates $(i)$ and $(iii).$\\
We set $\rho=d(E,F)/4t$ in (\ref{nadiag}),  we get
\[||\na e^{-tA}f ||_{L^2(F)}\le \frac{C}{\sqrt{t}}\Big(1+\frac{d(E,F)}{4\sqrt{t}}\Big) e^{-d^2(E,F)/8t} || f||_2.\]
This gives estimate $(ii)$.\qed
\end{proof}

Now, we study  the $L^p-L^2$ boundedness of the semigroup, of its gradient, and of $(V^{1/2}e^{-tA})_{t>0}$.

\begin{prop}\label{p2}
Suppose that $A\ge \e V$, then $(e^{-tA})_{t>0}$, $(\sqrt{t}\na e^{-tA})_{t>0}$ and $(\sqrt{t}V^{1/2}e^{-tA})_{t>0}$ are $L^p-L^2$ bounded for all $p\in(p'_0;2]$. Here $p'_0$ is the dual exponent of $p_0$ where $p_0=\frac{2N}{(N-2)\big(1-\sqrt{1-\frac{1}{1+\e}}\big)}$, and the dimension $N\ge 3$. More precisely we have for all $t>0$:
\begin{itemize}
 \item[i)] $\N e^{-tA}f\N_2\le Ct^{-\gamma_{p}}\N f\N_p$,
\item [ii)]$\N \sqrt{t}\na e^{-tA}f\N_2\le Ct^{-\gamma_{p}}\N f\N_p$,
\item[iii)]$\N \sqrt{t}V^{1/2} e^{-tA}f\N_2\le Ct^{-\gamma_{p}}\N f\N_p$,
\end{itemize}
 where $\gamma_{p}=\frac{N}{2}\big(\frac{1}{p}-\frac{1}{2}\big)$.
 \end{prop}

\begin{proof} \textit{i)} We apply the Gagliardo-Nirenberg inequality
 \[||u||^2_2\le C_{a,b} ||\na u||^{2a}_2||u||^{2b}_p,\] where $a+b=1$ and $(1+2\gamma_p)a=2\gamma_p$, to $u=e^{-tA}f$ for all $f\in L^2\cap L^p$, all $ t>0$, and all $p\in(p'_0;2]$.  We obtain
\[ ||e^{-tA}f||^2_2\le C_{a,b}||\na e^{-tA}f||^{2a}_2||e^{-tA}f||^{2b}_p.\]
At present we use the boundedness of the semigroup on $L^p$ for all $p\in (p'_0;2]$ proved in \cite{LSV}, and the fact that $||\na u||_2^2\le (1+1/\e)(Au,u)$ from the strong subcriticality condition, then we obtain that
\[||e^{-tA}f||^{2/a}_2\le-C\psi'(t)||f||_p^{2b/a}\]
where $\psi(t)=||e^{-tA}f||^2_2$.
This implies \[||f||^{-2b/a}_p\le C (\psi(t)^{\frac{a-1}{a}})'.\]
Since $ \frac{2b}{a}=\frac{1}{\gamma_p}$ and $\frac{a-1}{a}=-\frac{1}{2\gamma_p}$,
integration  between $0$ and $t$ yields
\[t||f||^{-1/\gamma_p}_p\le C||e^{-tA}f||^{-1/\gamma_p}_2,\]
which gives \textit{i)}.\\
We obtain \textit{ii)} by using the following decomposition:
\[\sqrt{t}\na e^{-tA}=\sqrt{t}\na A^{-1/2}A^{1/2} e^{-tA/2}e^{-tA/2},\]
the boundedness of $\na A^{-1/2}$ and of $(\sqrt{t}A^{1/2} e^{-tA})_{t>0}$ on $L^2$, and the fact that $(e^{-tA})_{t>0}$ is  $L^p-L^2$ bounded for all $p\in(p'_0;2]$ proved in \textit{i)}.\\
We prove \textit{iii)} by using the following decomposition:
\[\sqrt{t}V^{1/2} e^{-tA}=\sqrt{t}V^{1/2} A^{-1/2}A^{1/2} e^{-tA/2}e^{-tA/2},\]
the boundedness of $V^{1/2} A^{-1/2}$ and of $(\sqrt{t}A^{1/2} e^{-tA})_{t>0}$ on $L^2$, and the fact that $(e^{-tA})_{t>0}$ is  $L^p-L^2$ bounded for all $p\in(p'_0;2]$ proved in \textit{i)}.\qed
\end{proof}

We invest the previous results to obtain :

\begin{teo}\label{ehd}
Assume that $A\ge \e V$ then  $(e^{-tA})_{t>0}$, \  $(\sqrt{t}\na e^{-tA})_{t>0}$ \quad and\ \ \quad
$(\sqrt{t}V^{1/2}e^{-tA})_{t>0}$ satisfy $L^p-L^2$ off-diagonal estimates for all $p\in(p'_0;2]$. Here $p'_0$ is the dual exponent of $p_0$ where $p_0=\frac{2N}{(N-2)\big(1-\sqrt{1-\frac{1}{1+\e}}\big)}$, and the dimension $N\ge 3$. Then we have for all $t>0$, all $p\in(p'_0;2]$, all closed sets $E$ and $F$ of $\R^N$ and all $f\in L^2\cap L^p$ with supp$f\subseteq E$
\begin{itemize}
\item[i)]
 \begin{eqnarray}
\N e^{-tA}f\N_{L^2(F)}\le Ct^{-\gamma_{p}}e^{-\frac{cd^2(E,F)}{t}}\N f\N_p,\label{ode}
\end{eqnarray}
\item [ii)]
\begin{eqnarray}
 \N \sqrt{t}\na e^{-tA}f\N_{L^2(F)}\le Ct^{-\gamma_{p}}e^{-\frac{cd^2(E,F)}{t}}\N f\N_p,\label{odeg}
\end{eqnarray}
\item[iii)]
\begin{eqnarray}\label{odev}
 \N \sqrt{t}V^{1/2} e^{-tA}f\N_{L^2(F)}\le Ct^{-\gamma_{p}}e^{-\frac{cd^2(E,F)}{t}}\N f\N_p,
\end{eqnarray}
\end{itemize}
 where $\gamma_{p}=\frac{N}{2}\big(\frac{1}{p}-\frac{1}{2}\big)$ and $C, c$ are positive constants.
\end{teo}

\textbf{\textit{Remark:}} By duality, we deduce from (\ref{ode}) a $L^2-L^p$ off-diagonal estimate of the norm of the semigroup for all $p\in[2;p_0)$, but we cannot deduce from (\ref{odeg}) and (\ref{odev}) the same estimate of the norm of $\sqrt{t}\na e^{-tA}f$ and of $\sqrt{t}V^{1/2} e^{-tA}f$  because they are not selfadjoint. This affects the boundedness of Riesz transforms and of $V^{1/2}A^{-1/2}$ on $L^p$ for $p>2$.\\

\begin{proof} \textit{ i)} In the previous proposition we have proved that
\[\N e^{-tA}f\N_2\le Ct^{-\gamma_{p}}\N f\N_p\]
 for all $p\in(p'_0;2]$. This implies that for all $t>0$
\[\N \chi_Fe^{-tA}\chi_Ef\N_2\le Ct^{-\gamma_{p}}\N f\N_p\]
where $\chi_M$ is the characteristic function of $M$.
The $L^2-L^2$ off-diagonal estimate proved in the Proposition \ref{22} implies that
 \[\N \chi_Fe^{-tA}\chi_Ef\N_2\le e^{-d^2(E,F)/4t}\N f\N_2.\]
Hence we can apply the Riesz-Thorin interpolation theorem and we obtain the off-diagonal estimate (\ref{ode}).\\
Assertions \textit{ii)} and \textit{iii)}  are proved in a similar way. We use $L^2-L^2$ off-diagonal estimates of  Proposition \ref{22} and assertions \textit{ii)} and  \textit{iii)} of Proposition \ref{p2}.\qed
\end{proof}

\section{Boundedness of $\na A^{-1/2}$ and $V^{1/2}A^{-1/2}$ on $L^p$ for $p\in(p_0';2]$}
This section is devoted to the study of the boundedness of $V^{1/2}A^{-1/2}$ and Riesz transforms  associated to  Schr\"{o}dinger operators with negative and strongly subcritical potentials. We prove that  $\na A^{-1/2}$ and $V^{1/2}A^{-1/2}$ are bounded on $L^p(\R^N), N\ge3$, for all $p\in (p'_0;2]$, where $p_0'$ is the exponent mentioned in Theorem \ref{ehd}.\\

\begin{teo}\label{rt}
 Assume that $A\ge \e V$ , then  $ \na A^{-1/2}$ is bounded on $L^p(\R^N)$ for $N\ge3$, for all $p\in (p'_0;2]$ where  $p'_0=\Big(\frac{2N}{(N-2)\big(1-\sqrt{1-\frac{1}{1+\e}}\big)}\Big)'.$
\end{teo}


 To prove Theorem \ref{rt}, we prove that $\na A^{-1/2}$ is of weak type $(p,p)$ for all $p\in (p'_0;2)$ by using the following theorem of Blunck and Kunstmann  \cite{BK}. Then by the boundedness of $\na A^{-1/2}$  on $L^2$, and the Marcinkiewicz interpolation theorem, we obtain boundedness on $L^p$ for all $p\in (p'_0;2]$.  This result can also be deduced from Theorem \ref{ehd} together with Theorem 1.1 of \cite{bk3}.

\begin{teo}\label{bk}
Let $p\in[1;2)$. Suppose that $T$ is sublinear operator of strong type $(2,2)$, and let $(A_r)_{r>0}$ be a family of linear operators acting on $L^2$.

Assume that for $j\ge2$
\begin{eqnarray}\label{1}
\left(\frac{1}{\n 2^{j+1}B\n}\int_{C_j(B)}\n T(I-A_{r(B)})f\n^2\right)^{1/2}\le g(j) \left(\frac{1}{\n B\n}\int_B\n f\n^p\right)^{1/p},
\end{eqnarray}
and for $j\ge1$
\begin{eqnarray}\label{2}
 \left(\frac{1}{\n 2^{j+1}B\n}\int_{C_j(B)}\n A_{r(B)}f\n^2\right)^{1/2}\le g(j) \left(\frac{1}{\n B\n}\int_B\n f\n^p\right)^{1/p},
\end{eqnarray}
for all ball $B$  with radius $r(B)$ and all $f$ supported in $B$. If $\Sigma :=\sum g(j)2^{Nj}<\i$, then $T$ is of weak type $(p,p)$, with a bound depending only on the strong type $(2,2)$ bound of $T$, p, and $\Sigma$. \\Here $C_1=4B$ and $C_j(B)=2^{j+1}B \smallsetminus  2^jB$ for $ j\ge2$, where $\lambda B$ is the ball of radius $\lambda r(B)$ with the same center as $B$, and $\n \lambda B\n$ its Lebesgue measure. \\
\end{teo}

\begin{proof}[of Theorem \ref{rt}]
Let $T=\na A^{-1/2}$. We prove assumptions (\ref{1}) and (\ref{2})  with $A_r=I-(I-e^{-r^2A})^m$ for some $m>N/4-\gamma_p$, using arguments similar to Auscher \cite{a} Theorem 4.2.\\
Let us prove (\ref{2}). For $f$ supported in a ball $B$ (with radius $r$),
\begin{eqnarray*}
\frac{1}{\n 2^{j+1}B\n^{1/2}}\N A_rf\N_{L^2(C_j(B))}&=& \frac{1}{\n 2^{j+1}B\n^{1/2}}\Big |\Big| \sum_{k=1}^m\binom{m}{k}(-1)^{k+1}e^{-kr^2A}f\Big |\Big|_{L^2(C_j(B))}\\
&\le& \frac{1 }{\n 2^{j+1}B\n^{1/2}} \sum_{k=1}^m\binom{m}{k}C(kr^2)^{-\gamma_p}e^{\frac{-cd^2(B,C_j(B))}{kr^2}}\N f\N_p.
\end{eqnarray*}
for all $p\in (p'_0;2)$ and all $f\in L^2\cap L^p$  supported in $B$. Here we use the  $L^p-L^2$ off-diagonal estimates (\ref{ode})  for  $p\in(p'_0;2]$. Since $\gamma_p=\frac{N}{2}(\frac{1}{p}-\frac{1}{2})$ we obtain
\begin{eqnarray*}
  \Big(\frac{1}{\n 2^{j+1}B\n}\int_{C_j(B)} \n A_rf\n^2\Big)^{1/2}&\le& \frac{Cr^{-2\gamma_p}}{\n   2^{j+1}B\n^{1/2}}e^{\frac{-cd^2(B,C_j(B))}{mr^2}}\N f\N_p\\
&\le&C2^{-jN/2}e^{\frac{-cd^2(B,C_j(B))}{r^2}}\Big( \frac{1}{\n B\n}\int_B\n f\n^p\Big)^{1/p}.
\end{eqnarray*}
This yields, for $j=1$,
\[\Big(\frac{1}{\n 4B\n}\int_{4B} \n A_rf\n^2\Big)^{1/2}\le C2^{-N/2}\Big( \frac{1}{\n B\n}\int_B\n f\n^p\Big)^{1/p},\]
and for $j\ge 2$
\[\Big(\frac{1}{\n 2^{j+1}B\n}\int_{C_j(B)} \n A_rf\n^2\Big)^{1/2}\le C2^{-jN/2}e^{-c4^j}\Big( \frac{1}{\n B\n}\int_B\n f\n^p\Big)^{1/p}.\]
Thus assumption (\ref{2}) of Theorem \ref{bk} holds with $\sum_{j\ge 1} g(j)2^{jN}<\i$.\\
It remains to check the assumption (\ref{1}):\\
We know that \[\na A^{-1/2}f=C\int_0^\i \na e^{-tA}f\frac{dt}{\sqrt{t}}\]
then, using the Newton binomial, we get
\begin{eqnarray*}
\na A^{-1/2}(I-e^{-r^2A})^mf&=&C\int_0^\i \na e^{-tA}(I-e^{-r^2A})^mf\frac{dt}{\sqrt{t}}\\
&=&C\int_0^\i  g_{r^2}(t)\na e^{-tA}fdt
\end{eqnarray*}
where
 $$g_{r^2}(t)=\sum_{k=0}^m\binom{m}{k}(-1)^k\frac{\chi_{(t-kr^2>0)}}{\sqrt{t-kr^2}}.$$
  Hence, using the $L^p-L^2$ off-diagonal estimate (\ref{odeg}), we obtain for all $p\in(p'_0;2)$, all $j\ge 2$, and all $f\in L^2\cap L^p$  supported in $B$
\[\N \na A^{-1/2}(I-e^{-r^2A})^mf\N_{L^2(C_j(B))}\le C\int_0^\i \n g_{r^2}(t)\n t^{-\gamma_p-1/2}e^{-c4^jr^2/t}dt \N f\N_p.\]
We observe that (see \cite{a} p. 27)
\[\n g_{r^2}(t)\n \le \frac{C}{\sqrt{t-kr^2}}\quad \textrm{if} \quad kr^2<t\le (k+1)r^2\le (m+1)r^2\]
and
\[\n g_{r^2}(t)\n \le Cr^{2m}t^{-m-1/2}\quad \textrm{if} \quad t>(m+1)r^2.\]
This yields
\begin{eqnarray}
\N \na A^{-\frac{1}{2}}(I-e^{-r^2A})^mf\N_{L^2(C_j(B))}&\le&C  \sum_{k=0}^m\int_{kr^2}^{(k+1)r^2}\frac{t^{-\gamma_p-1/2}}{\sqrt{t-kr^2}}e^{-\frac{c4^jr^2}{t}}dt \N f\N_p \nonumber\\
&+&C\int_{(m+1)r^2}^\i r^{2m}t^{-\gamma_p-1-m}e^{-\frac{c4^jr^2}{t}}dt\N f\N_p\nonumber\\
&\le& I_1+I_2.\label{iiii}
\end{eqnarray}
We have
$$I_2:=C\int_{(m+1)r^2}^\i r^{2m}t^{-\gamma_p-1-m}e^{-\frac{c4^jr^2}{t}}dt\N f\N_p \le C r^{-2\gamma_p}2^{-2j(m+\gamma_p)}\N f\N_p,$$ by the Laplace transform formula, and
\begin{eqnarray*}
I_1&:=&C\N f\N_p\sum_{k=0}^m\int_{kr^2}^{(k+1)r^2}\frac{t^{-\gamma_p-1/2}}{\sqrt{t-kr^2}}e^{-\frac{c4^jr^2}{t}}dt \\
&=&C\N f\N_p\Big(\sum_{k=1}^m\int_{kr^2}^{(k+1)r^2}\frac{ t^{-\gamma_p-1/2}}{\sqrt{t-kr^2}}  e^{-\frac{c4^jr^2}{t}}dt +\int_0^{r^2}t^{-\gamma_p-1}e^{-\frac{c4^jr^2}{t}}dt\Big) \\
&=&J_1+J_2.
\end{eqnarray*}
In the preceding equation
\begin{eqnarray*}
 J_1&:=&C\N f\N_p\sum_{k=1}^m\int_{kr^2}^{(k+1)r^2}\frac{ t^{-\gamma_p-1/2}}{\sqrt{t-kr^2}}e^{-\frac{c4^jr^2}{t}}dt \\
&\le& C \N f\N_p e^{-\frac{c4^j}{m+1}}\sum_{k=1}^m(kr^2)^{-\gamma_p-1/2}\int_{kr^2}^{(k+1)r^2}(t-kr^2)^{-1/2}dt\\
&\le& Cr^{-2\gamma_p}2^{-2j(m+\gamma_p)}\N f\N_p,
\end{eqnarray*}
and
\begin{eqnarray*}
 J_2&:=&C\int_0^{r^2}t^{-\gamma_p-1}e^{-\frac{c4^jr^2}{t}}dt \N f\N_p\\
&\le&C \N f\N_p e^{-\frac{c4^j}{2(m+1)}}\int_0^{r^2}t^{-\gamma_p-1}e^{-\frac{c4^jr^2}{2t}}dt\\
&\le&C \N f\N_p 2^{-2jm}\int_0^{r^2}t^{-1-\gamma_p}C(2^{-2j}r^{-2}t)^{\gamma_p}e^{-\frac{c4^jr^2}{4t}}dt\\
&\le&C \N f\N_p 2^{-2j(m+\gamma_p)}r^{-2\gamma_p}\int_0^{r^2}t^{-1}e^{-\frac{c4^jr^2}{4t}}dt\\
&\le&C r^{-2\gamma_p} 2^{-2j(m+\gamma_p)}\N f\N_p.
\end{eqnarray*}
Here, for the last inequality, we use the fact that $j\ge2$ to obtain the convergence of the integral without dependence on $r$ nor on $j$.\\
We can therefore employ these estimates in (\ref{iiii}) to conclude that
$$\N\na A^{-1/2}(I-e^{-r^2A})^mf\N_{L^2(C_j(B))}\le C r^{-2\gamma_p} 2^{-2j(m+\gamma_p)}\N f\N_p,$$
which implies
 \[\Big(\frac{1}{\n 2^{j+1}B\n}\int_{C_j(B)} \n \na A^{-\frac{1}{2}}(I-e^{-r^2A})^mf\n^2\Big)^{\frac{1}{2}}\le  C2^{-2j(m+\gamma_p+\frac{N}{4})}\Big( \frac{1}{\n B\n}\int_B\n f\n^p\Big)^{\frac{1}{p}}\]
where $\sum g(j)2^{jN}<\i$ because we set $m>N/4-\gamma_p.$\qed
\end{proof}

\begin{prop}\label{v}
 Assume that $A\ge \e V$, then  $V^{1/2} A^{-1/2}$ is bounded on $L^p(\R^N)$ for $N\ge3$, for all $p\in (p'_0;2]$ where $p_0'$ is the dual exponent of $p_0$ with $p_0=\frac{2N}{(N-2)\big(1-\sqrt{1-\frac{1}{1+\e}}\big)}.$
\end{prop}

\begin{proof}
We have seen in (\ref{lv2}) that the operator $V^{1/2} A^{-1/2}$ is bounded on $L^2$. To prove its boundedness on $L^p$ for all $p\in(p_0';2]$ we prove that it is of weak type $(p,p)$ for all $p\in (p_0';2)$ by checking assumptions (\ref{1}) and (\ref{2}) of Theorem \ref{bk}, where $T=V^{1/2}A^{-1/2}$. Then, using the Marcinkiewicz interpolation theorem, we deduce boundedness on $L^p$ for all $p\in(p_0';2]$.\\
We check assumptions of Theorem \ref{bk} similarly as we did in the proof of Theorem \ref{rt}, using the $L^p-L^2$ off-diagonal estimate (\ref {odev}) instead of (\ref{odeg}).\qed
\end{proof}

Let us now move on, setting $V=c|x|^{-2}$ where $0<c< (\frac{N-2}{2})^2$, which is strongly subcritical thanks to the Hardy inequality,  we prove that the associated Riesz transforms are not bounded on $L^p$ for $p\in (1;p'_0)$ neither for $p\in (p_{0*}; \i)$. Here $p_{0*}=\frac{p_0N}{N+p_0}$ is the reverse Sobolev exponent of $p_0$. \\

\begin{prop}
Set $V$  strongly subcritical and $N\ge 3$. Assume that $\na A^{-1/2}$ is bounded on $L^p$ for some $p\in (1;p'_0)$. Then there exists an exponent $q_1\in [p;p'_0)$ such that $(e^{-tA})_{t>0}$ is  bounded on $L^{r}$ for all $r\in(q_1;2)$.
\end{prop}

 Consider now $V=c|x|^{-2}$ where $0<c< (\frac{N-2}{2})^2$. It is proved in \cite{LSV} that the semigroup does not act on $L^p$ for $p \notin (p'_0;p_0)$. Therefore we obtain from this proposition that the Riesz transform $\na A^{-1/2}$ is not bounded on $L^p$ for $p\in(1;p_0')$.\\

\begin{proof}
 Assume that $\na A^{-1/2}$ is bounded on $L^p$ for some $p\in (1;p'_0)$. By the boundedness on $L^2$ and the  Riesz-Thorin interpolation theorem, we get the boundedness of $\na A^{-1/2}$ on  $L^q$ for all $q\in [p;2]$. Now we apply the Sobolev inequality
\begin{eqnarray}\label{sobolev}
\N f \N_{q^*}\le C\N \na f \N_q
 \end{eqnarray}
where $q^*=\frac{Nq}{N-q}$ if $q<N$ to $f:=A^{-1/2}u$, so we get
\[\N A^{-1/2}u\N_{q^*}\le C\N \na A^{-1/2}u\N_q\le C\N u \N_q\]
for all $q\in [p;2]$. In particular, $\N A^{-1/2}\N_{q_1-q_1^*}\le C$ where $p\le q_1< p_0'$ such that $q_1^*>p_0'$.\\
Decomposing the semigroup as follows
\begin{eqnarray}\label{sem}
e^{-tA}=A^{1/2}e^{-tA/2}e^{-tA/2}A^{-1/2}\end{eqnarray}
where $A^{-1/2}$ is $L^{q_1}-L^{q_1^*}$ bounded 
, $e^{-tA/2}$ has $L^{q_1^*}-L^2$ norm bounded by $Ct^{-\gamma_{q_1^*}}$ (Proposition \ref{p2}) and $A^{1/2}e^{-tA/2}$ is $L^2-L^2$ bounded by $Ct^{-1/2}$ because of the analyticity of the semigroup on $L^2$. Therefore, we obtain
\[\N e^{-tA}\N_{q_1-2}\le Ct^{-\gamma_{q_1^*}-1/2}= Ct^{-\gamma_{q_1}}.\]
We now interpolate this norm with the $L^2-L^2$ off-diagonal estimate of the norm of $e^{-tA}$, as we did in the proof of Theorem \ref{ehd}, so we get a $L^{r}-L^2$ off-diagonal estimate for all $r\in(q_1;2)$. Then Lemma 3.3 of \cite{a} yields that
 $(e^{-tA})_{t>0}$ is  bounded on $L^{r}$ for all $r\in(q_1;2)$ for $q_1\in [p; p'_0)$ such that $q_1^*>p_0'$.\qed
\end{proof}

\begin{prop}\label{ce}
Set $V$  strongly subcritical and $N\ge 3$. Assume that
$\na A^{-1/2}$ is bounded on $L^p$ for some $p\in (p_{0*};\i)$. Then there exists an exponent $q_2> p_{0*}$  such that the semigroup $(e^{-tA})_{t>0}$ is  bounded on $L^s$ for all $s\in(2;q_2^*)$. Here $q_2^*>p_0$.
\end{prop}
Consider now $V=c|x|^{-2}$ where $0<c< (\frac{N-2}{2})^2$. It is proved in \cite{LSV} that the semigroup does not act on $L^p$ for $p \notin (p'_0;p_0)$. Therefore we obtain from this proposition that the Riesz transforms $\na A^{-1/2}$ are not bounded on $L^p$ for $p\in(p_{0*};\i)$.\\
\begin{proof}
Assume that $\na A^{-1/2}$ is bounded on $L^p$ for some $p\in (p_{0*};\i)$. Then by interpolation we obtain the boundedness of $\na A^{-1/2}$ on $L^q$ for all $q\in[2;p]$. In particular,
\[||\na A^{-1/2}||_{q_2-q_2}\le C\]
where $p_{0*}< q_2< p_0$, $q_2\le p$, $q_2< N$.
 Using the Sobolev inequality (\ref{sobolev}), we obtain that $A^{-1/2}$ is $L^{q_2}-L^{q_2^*}$ bounded where $q_2^*> p_0$.\\
Now we decompose the semigroup as follows
\begin{eqnarray}\label{se}e^{-tA}=A^{-1/2}e^{-tA/2}A^{1/2}e^{-tA/2}.
\end{eqnarray}
Thus we remark that it is $L^2-L^{q_2^*}$ bounded where $q_2^*> p_{0}$.\\ Then, using similar arguments as in the previous proof, we conclude that
 $(e^{-tA})_{t>0}$ is  bounded on $L^{s}$ for all $s\in(2;q_2^*)$ for  $p_{0*}<q_2< \text{inf}(p_0,p,N)$.\qed
\end{proof}

\section{Boundedness of $\na A^{-1/2}$ and  $V^{1/2}A^{-1/2}$ on $L^p$ for all $p\in(1;N)$}
\label{tk}

In this section we assume that $V$ is strongly subcritical in the Kato subclass $K_N^{\i}, N\ge 3$. Following Zhao  \cite{z}, we define\\
\[ K_N^{\i}:= \left\{ V\in K_N^{loc}; \lim_{B\uparrow \i} \left[ \sup_{x\in\R^N}\int_{|y|\ge B}\frac{|V(y)|}{|y-x|^{N-2}}dy \right] =0 \right\}, \]
where $K_N^{loc}$ is the class of potentials that are locally in the Kato class $K_N$.\\
 For necessary background of the Kato class see \cite{s} and references therein.\\

 We use results proved by stochastic methods to deduce a $L^1-L^\i$ off-diagonal estimate of the norm of the semigroup which leads to the boundedness of $\na A^{-1/2}$ and  $V^{1/2}A^{-1/2}$ on $L^p$ for all $p\in(1;N)$.\\

\begin{teo}\label{e.g}
Let $A$ be the Schr\"{o}dinger operator $-\Delta-V, V\ge 0$. Assume that $V$ is strongly subcritical in the class $K_N^\i, (N\ge 3),$
then  $\na A^{-1/2}$ and $V^{1/2}A^{-1/2}$ are of weak type $(1,1)$, they are bounded on $L^p$ for all $p\in(1;2]$. If in addition $V\in L^{N/2}$,
then $\na A^{-1/2}$ and $V^{1/2}A^{-1/2}$ are bounded on $L^p$ for all $p\in(1;N)$.
\end{teo}

\begin{proof}
We assume that $V$ is strongly subcritical in the class  $K_N^\i$. Therefore $V$ satisfies assumptions of Theorem 2 of \cite{t} (The classes $K_\i$ and $S_\i$ mentioned in \cite{t} are equivalent to the class $K_N^\i$ (see Chen \cite{ch} Theorem 2.1 and Section 3.1)).
Thus the heat kernel associated to $(e^{-tA})_{t>0}$ satisfies a Gaussian estimate. Therefore $(e^{-tA})_{t>0}$, $(\sqrt{t}\na e^{-tA})_{t>0}$, and $(\sqrt{t}V^{1/2}e^{-tA})_{t>0}$ satisfy $L^1-L^2$ off-diagonal estimates. Arguing now as in the proof  of Theorem \ref{rt} (or using Theorem 5 of \cite{sik}) we conclude that $\na A^{-1/2}$ and $V^{1/2}A^{-1/2}$ are of weak type $(1,1)$ and they are bounded on $L^p$ for all $p\in(1;2]$.

To prove the boundedness of $\na A^{-1/2}$ on $L^p$ for higher $p$ we use the Stein complex interpolation theorem (see \cite{sw} Section V.4).\\
Let us first mention that $D:= R(A)\cap L^1\cap L^\i$ is dense in $L^p$ for all $p\in(1;\i)$ provided that $V$ is strongly subcritical in $K_N^\i, N\ge3$. We prove the density as in \cite{aba}, where in our case we have  the following estimate
\begin{eqnarray}\label{k}
|f_k-f|\le k(c(-\D)+k)^{-1}f
\end{eqnarray}where $f_k:=A(A+k)^{-1}f$ and $c$ is a positive constant.
This estimate holds from the Gaussian estimate of the heat kernel associated to the semigroup $(e^{-tA})_{t>0}$.

Set $F(z):=<(-\D)^zA^{-z}f,g>$ where $f\in D$, $g\in C_0^\i(\R^N)$ and $z\in S:=\{x+iy\  \text{such that}\  \ x\in [0;1] \ \  \text{and}\ \  y \in \R^N\}.\ \ F(z)$ is admissible. Indeed, the function $z\longmapsto F(z)$ is continuous in $S$ and analytic in its interior. In addition
\begin{eqnarray}\label{az}
|F(z)|=|<A^{-z}f, (-\D)^{\overline{z}}g>|\le ||A^{-z}f||_2||(-\D)^{\overline{z}}g||_2.
\end{eqnarray}
For $\Re{\overline{z}}\in (0;1)$,  $D(-\D)\subset D((-\D)^{\overline{z}})$, so
\begin{eqnarray}\label{g}
||(-\D)^{\overline{z}}g||_2\le C ||g||_{W^{2,2}}
\end{eqnarray}
for all $z\in S$.\\
When $V$ is strongly subcritical, $A$ is non-negative self-adjoint operator on $L^2$, hence $||A^{iy}||_{2-2}\le1$ for all $y\in \R$. Therefore for all $z=x+iy\in S$ and $f=Au\in R(A)$ we have
\begin{eqnarray}\label{f}
||A^{-z}f||_2&\le& ||A^{-iy}||_{2-2}||A^{1-x}u||_2\nonumber\\
&\le& C(||u||_2+||Au||_2).
\end{eqnarray}
Here we use $D(A)\subset D(A^{1-x})$ because $(1-x)\in (0;1).$\\ 
Now we employ (\ref{g}) and (\ref{f}) in (\ref{az}) to deduce the admissibility of $F(z)$ in $S$. Thus we can apply the Stein complex interpolation theorem to $F(z)$.\\
Since $V$ is strongly subcritical and belongs to the class $K_N^\i, N\ge3$, we obtain a Gaussian estimate of the heat kernel of $A$.  Thus $A$ has a $H^\i$-bounded calculus on $L^p$ for all $p\in(1;\i)$ (see e.g. \cite{BK} Theorem 2.2). Hence
\[|F(iy)|\le ||A^{-iy}f||_{p_0}||(-\D)^{-iy}g||_{p_0'}\le C_{\gamma,p_0}e^{2\gamma|y|}||f||_{p_0}||g||_{p_0'}\]
for all $\gamma>0$, all $p_0\in(1;\i)$.\\
Let us now estimate $||VA^{-1}||_{p_1-p_1}$. By H\"{o}lder's inequality
\begin{eqnarray}\label{h}
||VA^{-1}u||_{p_1}\le||V||_{N/2}||A^{-1}u||_q
\end {eqnarray}
where $p_1<N$ and $\frac{1}{p_1}=\frac{1}{q}+\frac{2}{N}$. As mentioned above we have a Gaussian upper bound for the heat kernel. In particular \[||e^{-tA}||_{1-\i}\le Ct^{-N/2}\] for all $t>0$. Therefore $A^{-1}$ extends to a bounded operator from $L^s$ to $ L^q$ such that $s<\frac{N}{2}$ and $\frac{1}{s}=\frac{1}{q}+\frac{2}{N}$, and we have
\[||A^{-1}u||_q\le C||u||_s.\]
(see Coulhon \cite{c}). Thus $s=p_1$, $D(A)\subseteq D(V)$ and (\ref{h}) implies
\[||VA^{-1}||_{p_1-p_1}\le C\] where $C$ depends on $||V||_{N/2}$.
Hence we can estimate
\begin{eqnarray}\label{hol}
||(-\D)A^{-1}u||_{p_1}&=& ||(-\D-V+V)A^{-1}u||_{p_1}\nonumber\\
&\le& ||u||_{p_1}+||VA^{-1}u||_{p_1}\nonumber\\
&\le& C||u||_{p_1}
\end{eqnarray}
where $C$ depends on $||V||_{N/2}$.
We return to $F(z)$,
\begin{eqnarray*}
|F(1+iy)|&\le& ||(-\D)A^{-1}A^{-iy}f||_{p_1}||(-\D)^{-iy}g||_{p_1'}\\
&\le& ||(-\D)A^{-1}||_{p_1-p_1}||A^{-iy}f||_{p_1}||(-\D)^{-iy}g||_{p_1'}\\
&\le&C_{\gamma,p_1, ||V||_{N/2}} e^{2\gamma|y|}||f||_{p_1}||g||_{p_1'}
\end{eqnarray*}
for all $p_1\in (1; N/2)$ and all $\gamma>0$.

From the Stein interpolation theorem it follows that for all $t\in[0;1]$ there exists a constant $M_t$ such that
\[|F(t)|\le M_t||f||_{p_t}||g||_{p'_t}\]
where $\frac{1}{p_t}=\frac{1-t}{p_0}+\frac{t}{p_1}$. Setting $t=\frac{1}{2}$ and using a density argument we conclude that $\na A^{-1/2}$ is bounded on $L^p$ for all $p\in(1;N)$.\\

To prove boundedness of $V^{1/2}A^{-1/2}$ on $L^p$ we use the following decomposition
\[V^{1/2}A^{-1/2}=V^{1/2}(-\D)^{-1/2}(-\D)^{1/2}A^{-1/2}.\]
 Assuming $V\in L^{N/2}$ we have by H\"{o}lder's inequality
 \[||V^{1/2}u||_p\le||V^{1/2}||_N||u||_q\]
 where $p<N$ and $\frac{1}{p}-\frac{1}{q}=\frac{1}{N}$. Then by Sobolev inequality and the boundedness of Riesz transforms associated to the Laplace operator we obtain
 \begin{eqnarray}\label{ss}
 ||V^{1/2}u||_p\le C_{p,N,||V||_{N/2}}||\na u||_p\le C_{p,N,||V||_{N/2}}||(-\D)^{1/2} u||_p
 \end{eqnarray}
 for all $p\in(1;N).$ Thus if $V\in L^{N/2}$ we have for all $p\in(1;N)$
 \[||V^{1/2}(-\D)^{-1/2}||_{p-p}\le C.\]
 Using the boundedness of Riesz transforms associated to the Schr\"{o}dinger operator $A$ we have
 \[||(-\D)^{1/2}A^{-1/2}u||_p\le 
 C||u||_p\]
 for all $p\in(1;N)$.\\ 
 Therefore $V^{1/2}A^{-1/2}$ is bounded on $L^p$ for all $p\in (1;N)$ provided that $V$ is strongly subcritical in the class $K_N^\i\cap L^{N/2}, N\ge3$.
\qed \end{proof}

 \textbf{\textit{Example:}}  Set $N\ge 3$, and let us take potentials  $V$  in the Kato subclass $K_N\cap L^{N/2}$ such that $V\sim c|x|^{-\aa}$ when $x$ tends to infinity, where $\aa>2$. Suppose that $||V||_{\frac{N}{2}}$ is small enough.
 Let us  prove that these potentials are strongly subcritical, so we should prove that
\[||V^{1/2}u||_2^2\le C ||\na u||_2^2\] where $C<1$. This is (\ref{ss})
where $p=2$, and $C<1$ for $||V||_{\frac{N}{2}}$ is small enough.
 Hence these potentials are strongly subcritical. Z.Zhao \cite{z} proved that they are in the subclass $K_N^\i$. Hence they satisfy the assumptions of Theorem  \ref{e.g}. Then $\na (-\Delta-V)^{-1/2}$ and $V^{1/2}(-\Delta-V)^{-1/2}$ are bounded on $L^p$ for all $p\in(1;N).$\\

\textbf{Remarks:} 1)  The proof of the previous theorem shows that
\[||Vu||_{p_1}\le C||Au||_{p_1}\] and
\[||\D u||_{p_1}\le C||Au||_{p_1}\] for all $p_1\in (1;N/2)$.\\
2) If we consider $H=-\D+V$ a Schr\"{o}dinger operator with non-negative potential $V\in L^{N/2}$, we obtain by the previous arguments the $L^{p_1}$-boundedness of $VH^{-1}$ and $\D H^{-1}$ for all $p_1\in (1;N/2)$, and the $L^p$-boundedness of $V^{1/2}H^{-1/2}$ and $\na H^{-1/2}$ for all $p\in(1;N)$.

\section{Schr\"odinger operators on Riemannian manifolds}
Let $M$ be a non-compact complete Riemannian manifold of dimension $N\ge3$. Denote by $d\mu$ the Riemannian measure, $\rho$ the geodesic distance on $M$ and $\na$ the Riemannian gradient. Denote by $|.|$ the length in the tangent space, and by $\N.\N_p$ the norm in $L^p(M,d\mu)$.
Let $-\D$ be the positive self-adjoint Laplace-Beltrami operator on $M$. Take $V$ a  strongly subcritical positive potential on $M$, which means that there exists an $\e>0$ such that
\begin{eqnarray}\label{ssc}
\int_{M} Vu^2d\mu\le \frac{1}{1+\e}\int_{M}|\na u|^2d\mu.
\end{eqnarray}
and set $A:=-\D -V$ the associated Schr\"odinger operator on $M$. By the sesquilinear form method $A$ is well defined, non-negative, and $-A$ generates a bounded analytic semigroup $(e^{-tA})_{t>0}$ on $L^2(M)$.\\

As in $\R^N$, we have the $L^2(M)$-boundedness of $V^{1/2}A^{-1/2}$ and of the Riesz transforms $\na A^{-1/2}$ if and only if $V$ is strongly subcritical.\\

We remark that methods used in \cite{LSV} hold in manifolds. The semigroup $(e^{-tA})_{t>0}$ can be extrapolated to $L^p(M)$, and it is uniformly bounded for $p\in\Big(\big(\frac{2}{1-\sqrt{1-\frac{1}{1+\e}}}\big)'; \big(\frac{2}{1-\sqrt{1-\frac{1}{1+\e}}}\big)\Big)$. If in addition  the Sobolev inequality
\begin{eqnarray}\label{sm}
\N f\N_{L^{\frac{2N}{N-2}}(M)}\le C \N |\na f|\N_{L^2(M)}\end{eqnarray}
 for all $f\in C_0^\i(M)$ holds on $M$, then we obtain for all $t>0$
 \[\N e^{-tA}\N_{L^p(M)-L^{\frac{pN}{N-2}}(M)}\le Ct^{-1/p}\] for all $p\in\Big(\big(\frac{2}{1-\sqrt{1-\frac{1}{1+\e}}}\big)'; \big(\frac{2}{1-\sqrt{1-\frac{1}{1+\e}}}\big)\Big)$. Using the $L^2(M)-L^2(M)$ off-diagonal estimate we obtain as in \cite{LSV} the fact that $(e^{-tA})_{t>0}$ is bounded on $L^p(M)$  
 for all $p\in (p_0';p_0)$ where $p_0:=\frac{2N}{N-2}\frac{1}{1-\sqrt{1-\frac{1}{1+\e}}}$.\\ For  classes of manifolds satisfying (\ref{sm}) see \cite{sc2}. Note that (\ref{sm}) is equivalent to 
 the following Gaussian upper bound of the heat kernel $p(t,x,y)$ of the Laplace-Beltrami operator (see \cite{var} and \cite{dav})
\begin{eqnarray}\label{egm}
p(t,x,y)\le C t^{-N/2} e^{-c\rho^2(x,y)/t}\quad \forall x,y\in M, t>0.
\end{eqnarray}
 We say that $M$ is of homogeneous type if for all $x\in M$ and $r>0$
 \begin{eqnarray}\label{ht}
 \mu(B(x,2r))\le C \mu(B(x,r))
 \end{eqnarray}
 where $B(x,r):=\{y\in M\ \ \text{such that}\ \ \rho(x,y)\le r\}$.
 \\
We say that the $L^2$-Poincar\'e inequalities hold on $M$ if there exists a positive constant $C$ such that
\begin{eqnarray}\label{p}
\int_{B(x,r)}|f(y)-f_r(x)|^2d\mu (y)\le C r^2\int_{B(x,r)}|\na f(y)|^2d\mu(y)
\end{eqnarray}
for all $f\in C^\i_0(M), x\in M, r>0$, where $f_r(x):=\frac{1}{\mu(B(x,r)}\int_{B(x,r)}f(y)d\mu(y)$.\\
Saloff-Coste \cite{sc} proved that (\ref{ht}) and (\ref{p}) hold if and only if the heat kernel $p(t,x,y)$ satisfies the following Li-Yau estimate
\begin{eqnarray}\label{ly}
\frac{C e^{-c\rho^2(x,y)/t}}{\mu(B(x,\sqrt{t}))}\le p(t,x,y)\le \frac{C_1 e^{-c_1\rho^2(x,y)/t}}{\mu(B(x,\sqrt{t}))}.
\end{eqnarray}
Arguing as in the Euclidean case we obtain the following theorem

\begin{teo}
Let $M$ be a non-compact complete Riemannian manifold  of dimension $N\ge3$. Assume  (\ref{ssc}) and (\ref{sm}).
Then $(e^{-tA})_{t>0}$, \  $(\sqrt{t}\na e^{-tA})_{t>0}$ \quad and\ \ \quad
$(\sqrt{t}V^{1/2}e^{-tA})_{t>0}$ satisfy $L^p(M)-L^2(M)$ off-diagonal estimates for all $p\in(p'_0;2]$. Here $p'_0$ is the dual exponent of $p_0$ where $p_0=\frac{2N}{(N-2)\big(1-\sqrt{1-\frac{1}{1+\e}}\big)}$. Then we have for all $t>0$, all $p\in(p'_0;2]$, all closed sets $E$ and $F$ of $M$, and all $f\in L^2(M)\cap L^p(M)$ with supp$f\subseteq E$
\begin{itemize}
\item[i)]
$\N e^{-tA}f\N_{L^2(F)}\le Ct^{-\gamma_{p}}e^{-\frac{c\rho^2(E,F)}{t}}\N f\N_p,$
\item [ii)]
 $\N \sqrt{t}\na e^{-tA}f\N_{L^2(F)}\le Ct^{-\gamma_{p}}e^{-\frac{c\rho^2(E,F)}{t}}\N f\N_p,$
\item[iii)]
 $\N \sqrt{t}V^{1/2} e^{-tA}f\N_{L^2(F)}\le Ct^{-\gamma_{p}}e^{-\frac{c\rho^2(E,F)}{t}}\N f\N_p,$
\end{itemize}
 where $\gamma_{p}=\frac{N}{2}\big(\frac{1}{p}-\frac{1}{2}\big)$ and $C, c$ are positive constants.
\end{teo}
We invest these off-diagonal estimates as in the proof of Theorem \ref{rt} to obtain the following result

\begin{teo}
Let $M$ be a non-compact complete Riemannian manifold of dimension $N\ge3$. Assume (\ref{ssc}), (\ref{sm}) and (\ref{ht}). 
Then $V^{1/2}A^{-1/2}$ and $ \na A^{-1/2}$ are bounded on $L^p(M)$  for all $p\in (p'_0;2]$ where  $p'_0=\Big(\frac{2N}{(N-2)\big(1-\sqrt{1-\frac{1}{1+\e}}\big)}\Big)'.$
\end{teo}

  We say that the potential $V$ is in the class $K_\i(M)$, if for any $\e>0$ there exists a compact set $K\subset M$ and $\delta>0$ such that
\[\sup_{x\in M} \int_{K^c}G(x,y)|V(y)|d\mu(y)\le \e\]
 where $K^c:= M \smallsetminus K$, and for all measurable sets $B\subset K$ with $\mu(b)<\delta$,
\[\sup_{x\in M} \int_{B}G(x,y)|V(y)|d\mu(y)\le \e.\]
Here $G(x,y):=\int_{0}^\i p(t,x,y)dt$ is the Green function, and $p(t,x,y)$ is the heat kernel of the Laplace-Beltrami operator.
This class is the generalization of $K_N^\i$ to manifolds (see \cite{ch} Section 2).\\

Since (\ref{ht}) and (\ref{p}) imply the Li-Yau estimate (\ref{ly}), we can use Theorem 2 of \cite{t} and obtain a Gaussian upper bound of the heat kernel of $-\D-V$. Thus arguing as in the Euclidean case, we obtain the following result 

\begin{teo}
Let $M$ be a non-compact complete Riemannian manifold of dimension $N\ge3$, and let $A$ be the Schr\"{o}dinger operator  $-\Delta-V, 0\le V\in L^{N/2}(M)\cap K_\i$. Assume that for all ball $B$, $\mu(B(x,r))\ge Cr^N$.
Assume (\ref{ssc}), (\ref{ht}) and (\ref{p}). Then $\D (-\D-V)^{-1}$ and $V(-\D-V)^{-1}$ are bounded on $L^p(M)$ for all $p\in (1;N/2)$.
\end{teo}

Now using Theorem 2 of \cite{t} and Theorem 5 of \cite{sik}, then arguing as in the Euclidean case we obtain the following

\begin{teo}\label{ourt}
Let $M$ be a non-compact complete Riemannian manifold of dimension $N\ge3$, and let $A$ be the Schr\"{o}dinger operator  $-\Delta-V, 0\le V\in K_\i$. 
Assume (\ref{ssc}), (\ref{ht}) and (\ref{p}). 
 Then  $\na A^{-1/2}$ and $V^{1/2}A^{-1/2}$ are of weak type $(1,1)$, thus they are bounded on $L^p(M)$ for all $p\in(1;2]$.\\
 If in addition we assume that for all ball $B \ \ \mu(B(x,r))\ge Cr^N$, and for some $r\in(2;N]$, the Riesz transforms $\na (-\D)^{-1/2}$ are bounded on $L^r(M)$ 
 then $\na A^{-1/2}$ and $V^{1/2}A^{-1/2}$ are bounded on $L^p(M)$ for all $p\in(1;r)$ provided that $V\in L^{N/2}(M)$.
\end{teo}

\textit{Remark} Let $M$ be a non-compact complete Riemannian manifold of dimension $N\ge3$. Let $H=-\D+V$ be a Schr\"{o}dinger operator with non-negative potential $V\in L^{N/2}(M)$. 
 Assume  that for some $r\in(2;N]$, the Riesz transforms $\na (-\D)^{-1/2}$ are bounded on $L^p(M)$ for all $p\in (2;r)$ or for $p=r$. Assume also  (\ref{sm}). Then the heat kernel associated to $H$ satisfies (\ref{egm}). 
 Hence we obtain by the previous argument the 
  $L^p$-boundedness of $V^{1/2}H^{-1/2}$ and $\na H^{-1/2}$ for all $p\in(1;r)$.\\

Note that (\ref{ht}) and (\ref{p}) hold on manifolds with non-negative Ricci curvature (see \cite{LY}) as well as the boundedness on $L^p(M)$ for all $p\in (1,\i)$ of Riesz transforms associated to the Laplace-Beltrami operator (see \cite{b}). The Sobolev inequality  (\ref{sm}) is valid on manifolds with Ricci curvature bounded from below satisfying \[\inf_{x\in M}\mu(B(x,1))>0\] (see \cite{h} Theorem 3.14). Therefore manifolds with non-negative Ricci curvature satisfying $\inf_{x\in M}\mu(B(x,1))>0$ are a class of manifolds where Theorem \ref{ourt} holds.

We mention that  Carron, Coulhon and Hassell \cite{cch} proved that the Riesz transforms $\na (-\D)^{-1/2}$ are bounded on $L^p(M)$ for all $p\in (2;N)$ on smooth complete Riemannian manifolds of dimension $N\ge3$ which are the union of a compact part and a finite number of Euclidean ends. Ji, Kunstmann and Weber \cite{jkw} proved that this boundedness holds for all $p\in(1;\i)$, on the complete connected Riemannian manifolds whose Ricci curvature is bounded from below, if there is a constant $a>0$ with $\sigma(-\D)\subset\{0\}\cup[a;\i)$. They also give examples of manifolds that satisfy their conditions.  Auscher, Coulhon, Duong and Hofmann \cite{acdh} proved that on complete non-compact Riemannian manifolds satisfying assumption (\ref{ly}), the uniform boundedness of $(\sqrt{t}\na e^{-t(-\D)})_{t>0}$ on $L^q$ for some $q\in(2;\i]$ implies the boundedness on $L^p(M)$ of $\na (-\D)^{-1/2}$ for all $p\in(2;q)$. And we have equivalence if $(\sqrt{t}\na e^{-t(-\D)})_{t>0}$ is u
 niformly bounded on $L^r$ for all $r\in(2;q)$.\\

Therefore we deduce the following propositions using our previous theorem and the criterion of \cite{acdh}. We also use the fact that the semigroup $(e^{-t(-\D-V)})_{t>0}$ is bounded analytic on $L^p(M)$ for all $p\in(1;\i)$. This is true on manifolds where assumptions (\ref{ht}) and (\ref{p}) hold and when $V\in K_\i$ satisfying (\ref{ssc}) (see e.g. \cite{bk2} Theorem 1.1).

\begin{prop}
Let $M$ be a non-compact complete Riemannian manifold of dimension $N\ge3$. Assume that for all ball $B \ \ \mu(B(x,r))\ge Cr^N$, assume (\ref{ssc}), (\ref{ht}) and (\ref{p}), and assume that $V\in K_\i\cap L^{N/2}(M)$. If for some $r\in(2;N]$
\[\N |\na e^{-t(-\D)}|\N_{L^r(M)-L^r(M)}\le C/\sqrt{t}\] for all $t>0,$ 
then \[\N |\na e^{-t(-\D-V)}|\N_{L^p(M)-L^p(M)}\le C/\sqrt{t}\] for all $t>0$, all $p\in(1,r)$.
\end{prop}

\begin{prop}
Let $M$ be a non-compact complete Riemannian manifold of dimension $N\ge3$. Assume (\ref{sm}) 
 and assume that $V\in L^{N/2}(M)$. If  for some $r\in (2;N]$
\[\N |\na e^{-t(-\D)}|\N_{L^r(M)-L^r(M)}\le C/\sqrt{t}\]
for all $t>0$, 
then \[\N |\na e^{-t(-\D+V)} |\N_{L^p(M)-L^p(M)}\le C/\sqrt{t}\]
for all $t>0$, all $p\in(1,r)$.
\end{prop}

\begin{center}
\textbf{Acknowledgments}\end{center}
I would like to thank my Ph.D advisor E.M. Ouhabaz for his advice and C. Spina for  very helpful suggestions. My thanks also to P. Auscher who mentioned an error in an earlier version.

\end{document}